\documentclass[a4paper,12pt,twoside]{article}
\usepackage{amssymb,amsmath,amsthm,latexsym}
\usepackage{amsfonts}
\usepackage{graphicx,caption,subcaption}
\usepackage[pdftex,bookmarks,colorlinks=false]{hyperref}
\usepackage[hmargin=1.25in,vmargin=1.25in]{geometry}
\usepackage[mathscr]{euscript}

\usepackage[auth-sc-lg,affil-sl]{authblk}
\newtheorem{thm}{Theorem}[section]

\newtheorem{cor}[thm]{Corollary}
\newtheorem{defn}[thm]{Definition}
\newtheorem{lem}[thm]{Lemma}

\newtheorem{prob}[thm]{Problem}
\newtheorem{prop}[thm]{Proposition}

\title{\textbf{\sc On Integer Additive Set-Filtered Graphs}}

\author{N. K. Sudev\footnote{Corresponding author}}
\affil{\small Department of Mathematics\\ Vidya Academy of Science \& Technology \\ Thalakkottukara, Thrissur - 680501, Kerala, India.\\ E-mail: {\em sudevnk@gmail.com}}

\author{K. P. Chithra}
\affil{\small Naduvath Mana, Nandikkara,\\ Thrissur - 680301, Kerala, India.\\ E-mail: {\em srgerminaka@gmail.com}}

\author{K. A. Germina}
\affil{\small Mathematics Research Centre (Kannur University) \\ Mary Matha Arts \& Science College \\  Mananthavady, Wayanad - 670645, Kerala, India.\\ E-mail: {\em srgerminaka@gmail.com}}

\date{}

\begin{document}
\maketitle

\begin{abstract}
Let $\mathbb{N}_0$ denote the set of all non-negative integers and $\mathcal{P}(\mathbb{N}_0)$ be its power set. An integer additive set-labeling (IASL) of a graph $G$ is an injective function $f:V(G)\to \mathcal{P}(\mathbb{N}_0)$ such that the induced function $f^+:E(G) \to \mathcal{P}(\mathbb{N}_0)$ is defined by $f^+ (uv) = f(u)+ f(v)$, where $f(u)+f(v)$ is the sumset of $f(u)$ and $f(v)$. In this paper, we introduce the notion of a particular type of integer additive set-indexers called integer additive set-filtered labeling of given graphs and study their characteristics.
\end{abstract}

\vspace{0.25cm}


\section{Introduction}

For all  terms and definitions, not defined specifically in this paper, we refer to \cite{BM} and \cite{FH} and \cite{DBW}. Unless mentioned otherwise, all graphs considered here are simple, finite and have no isolated vertices.

The {\em sumset} of two non-empty sets $A$ and $B$,  denoted by  $A+B$, is defined as $A + B = \{a+b: a \in A, b \in B\}$. If either $A$ or $B$ is countably infinite, then their sumset $A+B$ is also countably infinite. Hence, all sets we mention here are finite. We denote the cardinality of a set $A$ by $|A|$. 
 
Using the concepts of sumsets, an integer additive set-labeling of a given graph $G$ is defined as follows.

\begin{defn}\label{D-IASL}{\rm
\cite{GA} Let $\mathbb{N}_0$ denote the set of all non-negative integers and $\mathcal{P}(\mathbb{N}_0)$ be its power set. An {\em integer additive set-labeling} (IASL, in short) of a graph $G$ is defined as an injective function $f:V(G)\to \mathcal{P}(\mathbb{N}_0)$ which induces a function $f^+:E(G) \to \mathcal{P}(\mathbb{N}_0)$ such that $f^+ (uv) = f(u)+ f(v),~ uv\in E(G)$. A graph which admits an IASL is called an {\em integer additive set-labeled graph} (IASL-graph).}
\end{defn}

\noindent The notion of an integer additive set-indexers of graphs was introduced in \cite{GA}.

\begin{defn}\label{D-IASI}{\rm
\cite{GA,GSO} An integer additive set-labeling $f:V(G)\to \mathcal{P}(\mathbb{N}_0)$ of a graph $G$ is said to be an {\em integer additive set-indexer} (IASI) if the induced function $f^+:E(G) \to \mathcal{P}(\mathbb{N}_0)$ defined by $f^+ (uv) = f(u)+ f(v)$ is also injective. A graph which admits an IASI is called an {\em integer additive set-indexed graph} (IASI-graph).}
\end{defn}

The existence of an integer additive set-labeling (or integer additive set-indexers) by a given graph was established in \cite{GS0} and the admissibility of integer additive set-labeling (or integer additive set-indexers) by given graph operations and graph products was established in \cite{CGS1}.

\begin{thm}
{\rm \cite{GS0}} Every graph $G$ admits an integer additive set-labeling (or an integer additive set-indexer).
\end{thm}

The cardinality of the set-label of an element (a vertex or an edge) of a graph $G$ is called the {\em set-indexing number} of that element. An element of $G$ having a singleton set-label is called a {\em mono-indexed element} of $G$.

In this paper, we study the characteristic of graphs which admit a certain type of integer additive set-labeling, called integer additive set-filtered labeling.

\section{Integer Additive Set-Filtered Graphs}

Note that all sets we consider in this paper are non-empty finite sets of non-negative integers. By the term a {\em ground set}, we mean a non-empty finite set of non-negative integers whose subsets are the set-labels of the elements of the given graph $G$. We denote the ground set used for labeling the elements of a graph $G$ by $X$. 

Motivated from the studies about topological IASL-graphs, made in \cite{GS12}, we study a set-labeling of a given graph, in which the collection of all set-labels of the vertices of a given graph forms a filter of the ground set used for the labeling. Let us first recall the definition of the filter of a set.

\begin{defn}\label{D-SF1}{\rm
\cite{KDJ1,JRM} Given a set $X$, a partial ordering $\subseteq$ can be defined on the power set $\mathcal{P}(X)$ by subset inclusion, turning $(\mathcal{P}(X),\subseteq)$ into a lattice. A {\em filter} on $X$, denoted by $\mathcal{F}$, is a non-empty subset of the power set $\mathcal{P}(X)$ of $X$ which has the following properties:

\begin{enumerate}\itemsep0mm
\item[(i)] $X \in \mathcal{F}$.
\item[(ii)] $A,B \in \mathcal{F} \implies A\cap B \in \mathcal{F}$. ($\mathcal{F}$ is closed under finite intersection).
\item[(iii)] $\emptyset \not \in \mathcal{F}$. ($\mathcal{F}$ is a proper filter).
\item[(iv)] $A \in \mathcal{F},  A \subset B, \implies B \in \mathcal{F}$ where $B$ is a non-empty subset of $X$. 
\end{enumerate} }
\end{defn}

In view of Definition \ref{D-SF1}, we define the notion of an integer additive set-filtered labeling of a given graph as follows.

\begin{defn}\label{D-IASFG}{\rm 
Let $X$ be a finite set of non-negative integers. Then, an integer additive set-labeling $f:V(G)\to \mathcal{P}(X)$ is said to be an {\em integer additive set-filtered labeling} (IASFL, in short) of $G$ if $\mathcal{F}=f(V)$ is a proper filter on $X$. A graph $G$ which admits an IASFL is called an {\em integer additive set-filtered graph} (IASF-graph). }
\end{defn}

Note that the null set can not be the set-label of any element of the graph $G$, with respect to an IASL defined on it. 

Does every given graph admit an integer additive set-filtered labeling? If not so, what is  the condition required for a graph to admit an IASFL? As answers to both questions, we establish a necessary and sufficient condition for an IASL $f$ of a given graph $G$ to be an IASFL of $G$ as follows. 

\begin{thm}\label{T-IASFL1}
An IASL $f$ defined on a given graph $G$ with respect to a non-empty ground set $X$ is an integer additive set-filtered labeling of $G$ if and only if the following conditions hold.
\begin{enumerate}\itemsep0mm
\item[(i)]  $0 \in X$.
\item[(ii)] every subset of $X$ containing $0$ is the set-label of some vertex in $G$. 
\item[(iii)] $0$ is an element of the set-label of every vertex in $G$.
\end{enumerate}
\end{thm}
\begin{proof}
Let $f$ be an IASFL defined on a given graph $G$, with respect to a non-empty set $X$. Then, $\mathcal{F}=f(V)$ is a filter on $X$. Therefore, $X\in \mathcal{F}$. Since, for any non-zero element $a\in X$, the sets $X$ and $X+\{a\}$ are of same cardinality, but indeed $X\subsetneq X+\{a\}$. Hence, $\{0\}$ must also be an element of $\mathcal{F}$. Hence, we notice that $0$ is an element of $X$.  Then, by condition (iv) of Definition \ref{D-SF1}, every subset of $X$ containing $0$ must belong to $\mathcal{F}$. For any two subsets $X_i$ and $X_j$ of $X$, $0\in X_i,\; 0\in X_j\implies 0\in X_i\cap X_j$ and hence $X_i\cap X_j$ also belongs to $\mathcal{F}$. If possible, let a set-label $X_i$ of a vertex $v_i$ of $G$ does not contain $0$. Then, $\{0\} \cap X_i=\emptyset$, which can not be the set-label of any vertex of $G$, contradicting the fact that $\mathcal{F}$ is a filter on $X$. Hence, no subset of $X$ which does not contain $0$, belongs to $\mathcal{F}$.

Conversely, assume that the set-label of every vertex of $G$ contains $0$ and every subset of $X$ containing $0$ is the set-label of some vertex of $G$. Since $0\in X, ~ X\in \mathcal{F}$. If $X_i$ and $X_j$ are the set-labels of two vertices in $G$, then both $X_i$ and $X_j$ contain the element $0$ and hence $X_i\cap X_j$ also contains $0$. Therefore, by the assumption, $X_i\cap X_j$ is also the set-label of some vertex in $G$. That is, $X_i, X_j \in \mathcal{F} \implies X_i\cap X_j \in \mathcal{F}$. As the set-label $X_i$ of any vertex $v_i$ of $G$ contains $0$, then every super set $X_j$ of $X_i$ also contains the element $0$. Therefore, by the hypothesis, $X_j$ is also the set-label of some vertex of $G$. That is, $X_i\in \mathcal{F}, ~ X_i \subset X_j \implies X_j\in \mathcal{F}$. Therefore, $\mathcal{F}$ is a filter on $X$. Hence, $f$ is an IASFL on $G$.
\end{proof}

\noindent From the above theorem we notice that all graphs do not possess IASFLs. Hence, a characterisation of the graphs that admit IASFLs arouses much interest. In view of Theorem \ref{T-IASFL1}, we now proceed to find the characteristics and properties of the graphs which admit IASFLs.

\noindent The following results is are immediate consequences of Theorem \ref{T-IASFL1}.

\begin{cor}
If a graph $G$ admits an IASFL, then $G$ has $2^{|X|-1}$ vertices.
\end{cor}
\begin{proof}
Note that $0\in X$ and let $|X|=n$. The number of $r$-element subsets of $X$ with a common element $0$ is $\binom{|X|-1}{r-1}$. Therefore, the number of subsets of $X$ containing the element $0$ is $\sum\limits_{i=0}^{n-1}\binom{n-1}{i}=2^{n-1}$. This completes the proof.
\end{proof}

\begin{cor}\label{C-IASFL2}
If a given graph $G$ admits an IASFL $f$, then only one vertex of $G$ can have a singleton set-label.
\end{cor}
\begin{proof}
Let $G$ be an IASF-graph. Then, by Theorem \ref{T-IASFL1}, $\{0\}$ is a set-label of some vertex in $G$. Let $a$ be a non-zero element in $X$. If $\{a\}$ is the set-label of some vertex of $G$, then the set $\{0\}\cap \{a\}=\emptyset$ must belong to $\mathcal{F}=f(V)$, which is a contradiction to Condition (iii) of Definition \ref{D-SF1}. Therefore, only one vertex of $G$ can have a singleton set-label. (That is, the only possible singleton set-label in $\mathcal{F}$ is $\{0\}$).
\end{proof}

Next, we establish the relation between the collection of the set-labels of vertices and the collection of the set-labels of the edges of an IASF-graph $G$ in the following result.

\begin{prop}\label{P-IASFL-ev}
If $f$ is an IASFL of a graph $G$, then $f^+(E(G))\subseteq f(V(G))$.
\end{prop}
\begin{proof}
If $u$ and $v$ are any two adjacent vertices of the IASF-graph $G$, then $f(u)$ and $f(v)$ contains $0$ and hence, being the sumset of $f(u)$ and $f(v)$, the set-label $f^+(uv)$ also contains the element $0$. Since every subset of $X$ containing $0$ is the set-label of some vertex in $G$, the set label of the edge $uv$ will also be a set-label of some vertex in $G$. Therefore, $f^+(E)\subseteq f(V)$.
\end{proof}

\noindent The following theorem is a consequence of Theorem \ref{T-IASFL1}.

\begin{thm}\label{T-IASFL2}
If a given graph $G$ admits an integer additive set-filtered labeling $f$, then every element of the collection $\mathcal{F}=f(V(G))$ belongs to some finite chain of sets in $\mathcal{F}$ of the form  $\{0\} =f(v_1)\subset f(v_2)\subset f(v_3) \subset \ldots\ldots \subset f(v_r)=X$. 
\end{thm}
\begin{proof}
Let $f$ be an IASFL defined on a graph $G$ and $\mathcal{F}$ be the collection of all set-labels of the vertices in $G$. Then, by Theorem \ref{T-IASFL1}, both $\{0\}$ and $X$ are in $\mathcal{F}$. Since every set-label in $\mathcal{F}$ contains $0$, $\{0\}$ is the subset of all set-labels in $\mathcal{F}$. Since $\mathcal{F} \subseteq \mathcal{P}(X)$, $X$ is the maximal set in $\mathcal{F}$ containing all sets in $\mathcal{F}$. Since $\mathcal{F}$ is a filter on $X$, if a subset $X_i$ of $X$ belongs to $\mathcal{F}$ implies every subset of $X$ containing $X_i$ is also in $\mathcal{F}$. Therefore, there exists some finite sequence $\{0\}\subset \ldots\ldots X_i\subset X_j \ldots\ldots X$ of subsets of $X$ in $\mathcal{F}$. Therefore, every set-label in $\mathcal{F}$ is contained in some finite chain of subsets of $X$ whose least element is $\{0\}$ and the maximal element is $X$.
\end{proof}

We have already identified the number of vertices required for a graph to admit an IASL with respect to a given ground set $X$. In this context, it is interesting to examine certain structural properties of a graph that admit an IASFL.  Hence, we have 

\begin{thm}\label{C-IASFL1}
If a graph $G$ admits an integer additive set-filtered labeling, with respect to a non-empty ground set $X$, then $G$ must have at least $2^{|X|-2}$ pendant vertices that are incident on a single vertex of $G$.
\end{thm}
\begin{proof}
Let $f$ be an IASFL defined on a given graph $G$. Then by Theorem \ref{T-IASFL1}, every subset of $X$ containing the element $0$ must belong to $\mathcal{F}$. Let $x_l$ be the maximal element of $X$. Then, for any non-zero element $x$ in $X$, $x+x_l\not\in X$.  
Therefore, if $X_l$ is a subset of $X$ containing $x_l$, then the vertex having $X_l$ as its set-label can not be adjacent to any vertex of $G$ other than the one that has the set-label $\{0\}$. Hence, all the subsets of $X$ containing $x_l$, including $X$ itself, can be adjacent only to the vertex having the set-label $\{0\}$. Note that the number of subsets of $X$ containing $0$ and $x_l$ is $2^{|X|-2}$. Therefore, the minimum number of pendant vertices in $G$ is $2^{|X|-2}$.
\end{proof}

Figure \ref{G-IASFG1} elucidates an IASF-graph with $2^{|X|-2}$ pendant vertices incident on a single vertex, where $X=\{0,1,2,3,4\}$.

\begin{figure}[h!]
\centering
\includegraphics[width=0.75\linewidth]{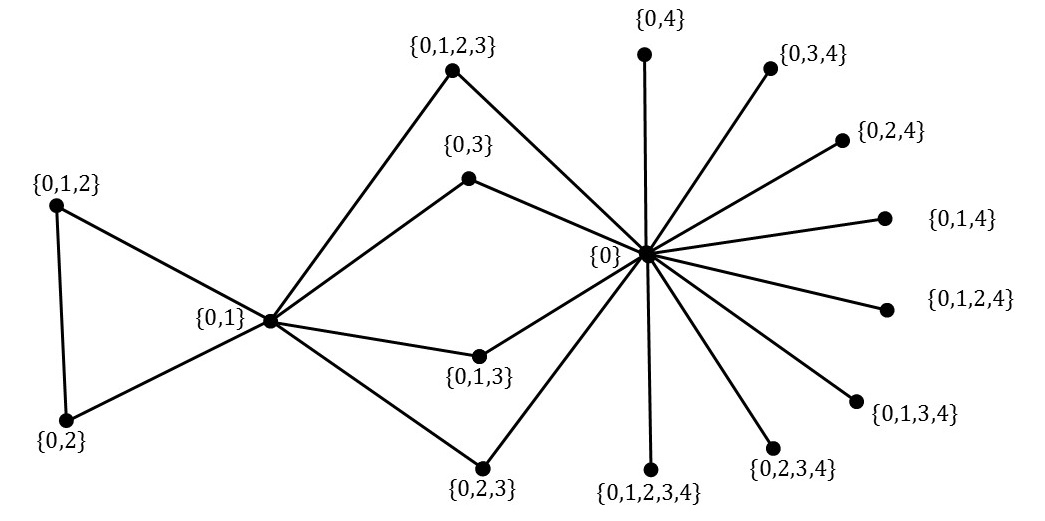}
\caption{An example to an IASF-graph}
\label{G-IASFG1}
\end{figure}

In view of the discussions we have made so far, we notice the following.

\begin{enumerate}\itemsep0mm
\item The existence of an IASFL is not a hereditary property. That is, an IASFL of a graph need not induce an IASFL for all of its subgraphs.
\item For $n\ge 3$, no paths $P_n$ admits an IASFL. 
\item No cycles admit IASFLs and as a result neither Eulerian graphs nor Hamiltonian graphs admit IASFLs.
\item Neither complete graphs nor complete bipartite graphs admit IASFLs. For $r>2$, complete $r$-partite graphs also do not admit IASFLs. 
\item Graphs having odd number of vertices never admits an IASFL.
\end{enumerate}

Another important property of IASFLs is that the existence of an IASFL is a {\em monotone} property. That is, removing any non-leaf edge of an IASFL graph preserves the IASFL of that graph.

\section{Relation Between Different IASLs}

In this section, let us verify the relation between an IASFL of a graph $G$ with certain other types of IASLs of $G$. First recall the definition of an exquisite IASL of a given graph $G$.

\begin{defn}{\rm
\cite{GS12} An {\em exquisite integer additive set-labeling} (EIASL, in short) is defined as an integer additive set-labeling $f:V(G)\to \mathcal{P}(\mathbb{N}_0)$ with the induced function $f^+:E(G) \to \mathcal{P}(\mathbb{N}_0)$ defined by $f^+ (uv) = f(u)+ f(v),~ uv\in E(G)$, holds the condition $f(u),f(v)\subseteq f^+(uv)$ for all adjacent vertices $u, v\in V(G)$. }
\end{defn}

The following theorem is a necessary and sufficient condition for an IASL of a graph $G$ to be an EIASL of $G$.

\begin{thm}\label{T-EIASL}
{\rm \cite{GS12}} Let $f$ be an IASL of a given graph $G$. Then, $f$ is an EIASL of $G$ if and only if $0$ is an element in the set-label of every vertex in $G$. 
\end{thm}

Invoking Theorem \ref{T-EIASL}, we establish the following relation between an IASFL and an exquisite IASL of a given graph $G$.

\begin{prop}
Every IASFL of a graph $G$ is also an exquisite IASL of $G$.
\end{prop}
\begin{proof}
Let $f$ be an IASFL of a given graph $G$. Then, by Theorem \ref{T-IASFL1}, the set-label of every vertex of $G$ contains $0$. Then by theorem \ref{T-EIASL}, $f$ is also an exquisite IASL of $G$.
\end{proof}

It is to be noted that, for an exquisite IASL $f$ of a graph $G$, $f(V)$ need not contain all the subsets of the ground set $X$ containing $0$. Therefore, every exquisite IASL of a graph $G$ need not be an IASFL of $G$.

Figure \ref{G-EIASL1} depicts a topological IASI of a graph $G$  with respect to the ground set $X=\{0,1,2,3,4\}$, which is not an IASFL of $G$.

\begin{figure}[h!]
\centering
\includegraphics[width=0.65\linewidth]{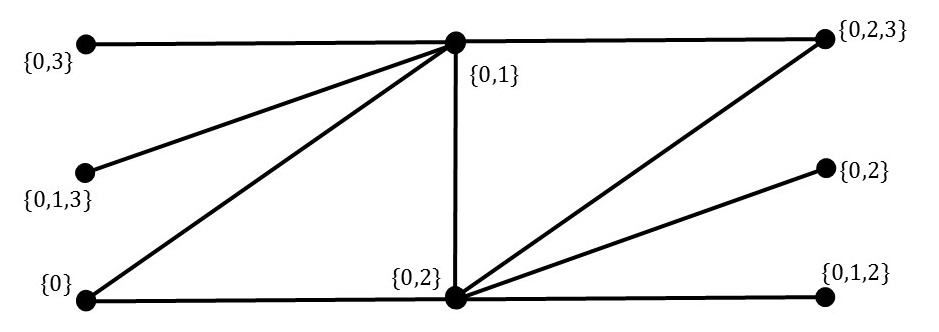}
\caption{An example to a  strong IASL of $G$ which is not an IASFL of $G$.}
\label{G-EIASL1}
\end{figure}

Let us now consider the notions of integer additive set-graceful graphs and integer additive set-sequential graphs, which are defined as follows. 

\begin{defn}{\rm
\cite{GS14,GS15} Let $f:V(G)\to \mathcal{P}(X)-\{\emptyset\}$ be an IASL defined on a graph $G$. Then, $f$ is called an {\em integer additive set-graceful labeling} (IASGL) of $G$ if  $f^{+}(E(G))=\mathcal{P}(X)-\{\emptyset,\{0\}\}$ and $f$ is called an {\em integer additive set-sequential labeling} (IASSL) of $G$ if  $f(V(G))\cup f^{+}(E(G))=\mathcal{P}(X)-\{\emptyset,\{0\}\}$. }
\end{defn} 

The following result checks whether an IASFL of a given graph $G$ can be an IASGL of the graph $G$.

\begin{prop}
No IASFL defined on a given graph $G$ is an IASGL of $G$.
\end{prop}
\begin{proof}
Let $f$ be an IASFL defined on $G$. By Proposition \ref{P-IASFL-ev}, $f(E(G))\subseteq f(V(G))$. Hence, set-labels of all edges of $G$ also contain the element $0$. That is, any subset $X_r$ of $X$ that does not contain $0$ will not be in $f(E(G))$. Therefore, $f(E(G))\neq \mathcal{P}(X)-\{\{0\},\emptyset\}$. Hence, $f$ is not an IASGL of $G$. 
\end{proof}

\noindent The following results can also be proved in a similar manner. 

\begin{prop}
No IASFL defined on a given graph $G$ is an IASSL of $G$.
\end{prop}
\begin{proof}
We have already proved that the set-labels of all elements of an IASF-graph $G$ contain the element $0$. Therefore, the set $f(V)\cup f^+(E)$ contains only those subsets of $X$ which contain $0$. That is, $f(V(G))\cup f^{+}(E(G))\ne \mathcal{P}(X)-\{\emptyset,\{0\}\}$. Hence, $f$ is not an IASSL of $G$.
\end{proof}

Another important IASL known to us, is a topological IASL, which is defined in \cite{GS13} as follows.

\begin{defn}{\rm
\cite{GS13} An integer additive set-labeling $f:V(G)\to \mathcal{P}(X)-\{\emptyset\}$  is called a {\em topological integer additive set-labeling} (TIASL) of $G$ if  $f(V(G))\cup \{\emptyset\}$ is a topology of $X$.}
\end{defn}

Can an IASFL of a given graph $G$ be a topological IASL of $G$? A relation between an IASFL and an TIASL of a graph $G$ is established in the following result.

\begin{prop}
Every IASFL of a graph $G$ is also a topological IASL of $G$.
\end{prop}
\begin{proof}
Let $f$ be an IASFL of a given graph $G$, with respect to a non-empty set $X$. Then $\mathcal{F}=f(V(G))$ is a filter on $X$. Let $\mathscr{T}=\mathcal{F}\cup \{\emptyset\}$. To show that $f$ is a TIASL of $G$, we need to show that $\mathscr{T}$ is a topology on $X$. Since $X\in \mathcal{F}$, we have $\emptyset, X \in \mathscr{T}$. Since $X$ is a finite set and $\mathcal{F}$ contains all subsets of $X$ consisting of $0$, the union of any number of elements of $\mathscr{T}$ is a set containing the element $0$ and hence belongs to $\mathscr{T}$. Similarly, the intersection of any two sets in $\mathcal{F}$ contains at least one element $0$ and hence the intersection of any number of elements in $\mathscr{T}$ is also in $\mathscr{T}$. Therefore, $\mathscr{T}$ is a topology on $X$. This completes the proof.
\end{proof}

If an IASL $f$ of a graph $G$ is an IASFL of $G$, then $f(V(G))$ contains only those subsets of $X$ consisting of the element $0$ and hence not all topological IASLs of $G$, with respect to $X$, can be the IASFLs of $G$. 

Figure \ref{G-TIASL1} depicts a topological IASI of a graph $G$ which is not an IASFL of $G$ with respect to the ground set $X=\{0,1,2,3\}$. 

\begin{figure}[h!]
\centering
\includegraphics[width=0.6\linewidth]{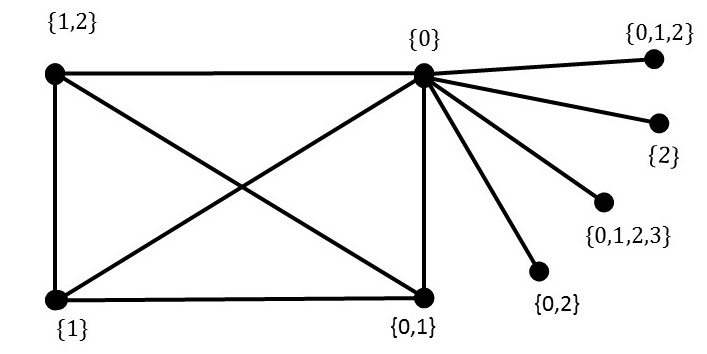}
\caption{An example to a TIASL of $G$ which is not an IASFL of $G$.}
\label{G-TIASL1}
\end{figure}

Another important type IASL of an given graph $G$ is a weak IASL of $G$, which is defined as follows.

\begin{defn}{\rm
\cite{GS1} A {\em weak integer additive set-labeling} (WIASL) of a graph $G$ is an IASL $f$ such that $|f^+(uv)| = \max(|f(u)|,|f(v)|)$ for all $u, v \in V(G)$.} 
\end{defn}

The following is a necessary and sufficient condition for an IASL to be a weak IASL of a given graph $G$.

\begin{lem}\label{L-WIASL}
{\rm \cite{GS1}} Let $f$ be an IASL defined on a given graph $G$. Then, $f$ is a WIASL of $G$ if and only if at least one end vertex of every edge of $G$ is mono-indexed.  
\end{lem}

An interesting question in this context is whether an IASFL of a given graph $G$ can be a weak IASL. The following result provides an answer to this question.

\begin{prop}
An IASFL of a graph $G$ is a weak IASL of $G$ if and only if $G$ is a star.
\end{prop}
\begin{proof}
Let $G=K_{1,n}$, where $n=2^{|X|-1}-1$, $X$ being the ground set that is used for set-labeling and let $f$ be an IASFL defined on $G$. Then, label the vertex $v$ at the centre of $G$ by the set $\{0\}$ and other vertices by other subsets of $X$ containing $0$. Therefore, every edge of $G$ has one mono-indexed end vertex. Hence by Lemma \ref{L-WIASL}, $f$ is a weak IASI of $G$.

Conversely, assume that an IASFL $f$ of $G$ is a weak IASI of $G$. Then, by Lemma \ref{L-WIASL}, every edge of $G$ must have at least one mono-indexed end vertex. But by Corollary \ref{C-IASFL2}, the only singleton set-label in $\mathcal{F}$ is $\{0\}$. Therefore, the vertex,say $v$, having the set-label $\{0\}$ must be adjacent to all other vertices of $G$ and the graph $G-v$ is a trivial graph. Therefore, $G$ is a star. 
\end{proof}

\noindent Next, recall the definition of a strong IASL of a given graph $G$.

\begin{defn}{\rm
\cite{GS2} A {\em strong integer additive set-labeling} (SIASL) of $G$ is an IASL such that if $|f^+(uv)| = |f(u)|\,|f(v)|$ for all $u, v \in V(G)$.} 
\end{defn}

The {\em difference set} of a set $A$ is the set of all positive differences between the elements of $A$. The difference set of a set $A$ is denoted by $D_A$.
	
Then, the following result is a necessary and sufficient condition for an IASL (or IASI) to be a SIASL (or SIASI) of a given graph $G$.

\begin{lem}\label{L-SIASL}
{\rm \cite{GS2}} Let $f$ be an IASL defined on a given graph $G$. Then, $f$ is a SIASL of $G$ if and only if the difference sets of any two adjacent vertices of $G$ are disjoint.
\end{lem}

\noindent Can a given IASFL $f$ of a given graph $G$ be a strong IASL of $G$? We know that $f$ is a strong IASL of $G$ if the difference sets of the set-labels of any two adjacent vertices of $G$ are disjoint. Using this result, we wish to verify whether there is any relation between an IASFL and a strong IASL of $G$.

\noindent Invoking Lemma \ref{L-SIASL} and Theorem \ref{T-IASFL1}, we propose the following result.

\begin{prop}
If an IASFL $f$  of a graph $G$ is a strong IASL of $G$, then $f(u)\cap f(v) = \{0\}$, where $u$ and $v$ of $G$ are two adjacent vertices of $G$.
\end{prop}
\begin{proof}
Assume that $f$ is an IASFL defined on a graph $G$. Let $u$ and $v$ be two adjacent vertices of $G$. Now, assume that $f$ is a strong IASL. Then, by Lemma \ref{L-SIASL}, $D_{f(v_i)}\cap D_{f(v_i)}=\emptyset$. If $f(u)$ and $f(v)$ have a common non-zero element, say $a$, then both $D_{f(v_i)}$ and $D_{f(v_i)}$ also contain the element $a$, contradicting the fact that $f$ is a strong IASL. Therefore, the set-labels of any two adjacent vertices have only one common element $0$. 
\end{proof}

It can be noted that the conditions $f(u)\cap f(v)=\{0\}$ and $D_{f(v_i)}\cap D_{f(v_i)}=\emptyset$, even together, do not produce the idea that every subset of $X$ containing $0$ is the set-label of some vertex of $G$. Therefore, every strong IASL of $G$ need not be an IASFL of $G$. 

Figure \ref{G-SIASL1} depicts a strong IASL of a graph $G$, with respect to the ground set $\{0,1,2,3\}$, which is not an IASFL of $G$.

\begin{figure}[h!]
\centering
\includegraphics[width=0.55\linewidth]{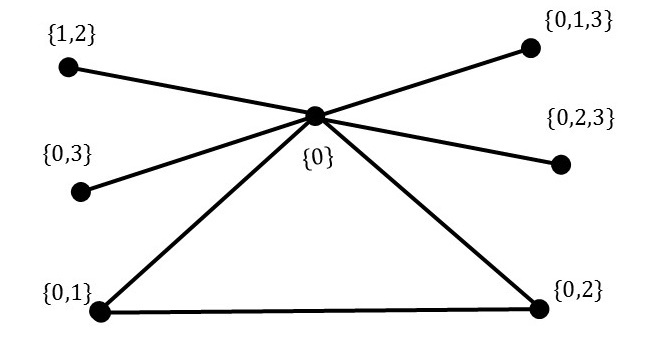}
\caption{An example to a  strong IASL of $G$ which is not an IASFL of $G$.}
\label{G-SIASL1}
\end{figure}  

Another type IASL which remains to be considered in this occasion is an arithmetic IASL of a graph $G$. An arithmetic IASL of a graph $G$ is an IASL $f$, with respect to which, the set-labels of all elements of $G$ are AP-sets. (An AP-set is a set whose elements are in an arithmetic progression). Since an AP-sets must have at least three elements, an IASFL of a graph $G$ can not be an Arithmetic IASL.

\section{Conclusion}
In this paper, we have introduced a new type of integer additive set-labeling and called an integer additive set-filtered labeling and have discussed certain characteristics and structural properties of graphs which admit this type of IASL. We have also discussed the relations, if any, with the other known types of IASLs. There are several other problems in this area are still open. The following are some of the problems we have identified in this area which need further investigation.

\begin{prob}
Determine a necessary and sufficient condition for an integer additive set-filtered labeling of a given graph $G$ to be an integer additive set-filtered indexer of $G$.
\end{prob}

\begin{prob}
Characterise the graphs which admit integer additive set-filtered indexers.
\end{prob}

\begin{prob}
Check the admissibility of IASFL by different operations and products of IASF-graphs.
\end{prob}

\begin{prob}
Check the admissibility of IASFL by the complement of IASF-graphs.
\end{prob}

\begin{prob}
Check the admissibility of IASFL by different certain graph classes.
\end{prob}

\begin{prob}
Check the admissibility of an induced IASFL by certain associated graphs such as line graphs, total graphs, subdivisions, homeomorphic graphs etc. of given IASF-graphs.
\end{prob}

An IASL (or IASI) is said to be {\em $k$-uniform} if $|f^+(e)| = k$ for all $e\in E(G)$. That is, a connected graph $G$ is said to have a $k$-uniform IASL (or IASI) if all of its edges have the same set-indexing number $k$. 

\begin{prob}
Determine the conditions required for an IASFL of a given graph to be a uniform IASFL.
\end{prob}

Studies on certain other types of integer additive set-labeling of graphs, both uniform and non-uniform, seem to be much promising. The integer additive set-labelings under which the vertices of a given graph are labeled by different standard sequences of non negative integers, are also worth studying. All these facts highlight a wide scope for further studies in this area.

\end{document}